\documentclass[12pt]{amsart}

\usepackage{amsmath,amssymb,amsthm,graphicx,color,amscd}
\usepackage[usenames,dvipsnames]{xcolor}
\usepackage{enumitem}

\usepackage[small,nohug,heads=vee]{diagrams}

\usepackage{hyperref}

\numberwithin{equation}{section}
\hypersetup{bookmarksdepth=3}

\theoremstyle{plain}
\newtheorem{theorem}[equation]{Theorem}

\newtheorem{corollary}[equation]{Corollary}

\newtheorem{lemma}[equation]{Lemma}

\theoremstyle{definition}
\newtheorem{definition}[equation]{Definition}

\theoremstyle{remark}
\newtheorem{remark}[equation]{Remark}

\newcommand{\al}{\alpha}

\newcommand{\be}{\beta}
\newcommand{\ben}{\begin{enumerate}}
\newcommand{\bit}{\begin{itemize}}

\newcommand{\een}{\end{enumerate}}
\newcommand{\eit}{\end{itemize}}

\newcommand{\fg}{\mathfrak{g}}
\newcommand{\fh}{\mathfrak{h}}

\newcommand{\ga}{\gamma}
\newcommand{\Ga}{\Gamma}

\newcommand{\h}{\mathcal{H}}
\renewcommand{\H}{\mathbb{H}}
\newcommand{\I}{\mathcal{I}}

\newcommand{\im}{\operatorname{Im}}

\renewcommand{\Im}{\operatorname{Im}}
\newcommand{\J}{\mathcal{J}}

\newcommand{\la}{\lambda}
\newcommand{\La}{\Lambda}

\newcommand{\loc}{\operatorname{loc}}
\newcommand{\lra}{\longrightarrow}

\newcommand{\om}{\omega}
\newcommand{\Om}{\Omega}

\newcommand{\R}{\mathbb{R}}
\newcommand{\ra}{\rightarrow}

\definecolor{gray}{gray}{0.7}

\newcommand{\restr}{\mbox{\Large \(|\)\normalsize}}

\newcommand{\Span}{\operatorname{span}}
\newcommand{\spt}{\operatorname{spt}}

\renewcommand{\th}{\theta}

\newcommand{\we}{\wedge}

\newcommand{\wt}{\operatorname{wt}}
\newcommand{\Z}{\mathbb{Z}}

\def\XXint#1#2#3{{\setbox0=\hbox{$#1{#2#3}{\int}$ }
\vcenter{\hbox{$#2#3$ }}\kern-.6\wd0}}

\newcommand{\cald}{\mathcal{D}}
\newcommand{\calr}{\mathcal{R}}
\begin{document}

\title{Sobolev mappings and the Rumin complex}

\author{Bruce Kleiner}
\thanks{BK was supported by NSF grants DMS-1711556, DMS-2005553, and a Simons Collaboration grant.}
\author{Stefan M\"uller}
\thanks{SM has been supported by the Deutsche Forschungsgemeinschaft (DFG, German Research Foundation) through
the Hausdorff Center for Mathematics (GZ EXC 59 and 2047/1, Projekt-ID 390685813) and the 
collaborative research centre  {\em The mathematics of emerging effects} (CRC 1060, Projekt-ID 211504053).  This work was initiated during a sabbatical of SM at the Courant Institute and SM would like to thank  R.V. Kohn and the Courant Institute
members and staff for 
their  hospitality and a very inspiring atmosphere.}
\author{Xiangdong Xie}
\thanks{XX has been supported by Simons Foundation grant \#315130.}

\maketitle

\begin{abstract}
We consider contact manifolds equipped with Carnot-Caratheodory metrics, and   show that the Rumin complex   is respected by Sobolev mappings: Pansu pullback induces a chain mapping between the smooth Rumin complex and the distributional Rumin complex.  As a consequence, the Rumin flat complex -- the analog of the Whitney flat complex in the setting of contact manifolds -- is bilipschitz invariant.   

We also show that for Sobolev mappings between general Carnot groups, Pansu pullback induces a chain mapping when restricted to a certain differential ideal $\J*\subset\Om^*G$ of the de Rham complex. 

Both results are applications of the Pullback Theorem from our previous paper. 
\end{abstract}

\tableofcontents

\section{Introduction}

This is a continuation of a series of papers on geometric mapping theory in Carnot groups, which is concerned with (partial) rigidity and (partial) regularity of bilipschitz, quasiconformal, or more generally Sobolev mappings between open subsets of Carnot groups  \cite{KMX1,KMX2,kmx_approximation_low_p}.  In \cite{KMX1} we showed that for Sobolev maps between Carnot groups,  the  pullback of differential forms using the Pansu differential partially respects exterior differentiation.  This Pullback Theorem was the starting point for a series of applications, see \cite{KMX1,KMX2,kmx_approximation_low_p}; in this paper we give some further applications.  We refer the reader to the introduction of \cite{KMX1} for more discussion of background and history.

Let $G$ be a step $s$ Carnot group of dimension $N$, homogeneous dimension $\nu$, and Lie algebra $\fg$ with grading $\fg=\oplus_jV_j$.  We fix a graded basis $\{X_i\}_{i\in I}$ for $\fg$, and let $\{\th_i\}_{i\in I}$ be the dual basis.  We let $I_j:=\{i\in I\mid X_i\in V_j\}$ and $I_{\geq j}:=\cup_{k\geq j}I_k$.  For $J\subset I$, we let $\th_J=\La_{i\in J}\th_i$ (we assume a linear ordering on $I$ so that this is well-defined).

Fixing an open subset $U\subset G$, we denote the de Rham complex on $U$ by $(\Om^*U,d)$. Inspired by Rumin \cite{rumin_thesis},   we let  $\I^*U\subset \Om^*U$  be the differential ideal generated by the $1$-forms $\al\in\Om^1U$ which vanish on the horizontal subbundle $V_1\subset TU$, and $\J^*U$ be the annihilator of $\I^*U$: 
\begin{equation}
\label{eqn_def_j}
\J^*U=\{\al\in\Om^*U\mid \al\we\be =0\;\;\text{for all}\;\;\be\in \I^*U\}\,.
\end{equation}
Hence 
$$
\I^*U=\Span\{\al_i\we \th_i+\be_i\we d\th_i\mid i\in I_{\geq 2}\,,\;\al_i,\be_i\in \Om^*U\}\,,
$$
and
$$
\J^*U=\{\al\in\Om^*U\mid \al\we\th_i=\al\we d\th_i=0\;\;\text{for all}\;\; i\in I_{\geq 2}\}\,.
$$

We now let $G'$ be another Carnot group and denote the associated objects with primes.

We recall that if $U\subset G$ is open and $f:G\supset U\ra G'$ is a $W^{1,p}_{\loc}$-mapping for some $p>\nu$, then $f$ is Pansu differentiable almost everywhere \cite{vodopyanov_differentiability_2003,kmx_approximation_low_p} (see also \cite{pansu,margulis_mostow_differential_quasiconformal_mapping}), and we may define the Pansu pullback  of a differential form $\om\in \Om^*G'$ by 
$$
f_P^*\om(x)=(D_Pf(x)^*\om)(f(x))\,,
$$
for a.e. $x\in U$.

Our first result is that Pansu pullback is compatible with the exterior derivative when restricted to forms from the differential ideal $\J^*G'$ on the target $G'$.

\begin{theorem}~\label{thm_j_theorem}
Let  $U \subset G$ be open, and $f:U\ra G'$ be a $W^{1,p}_{\loc}$-mapping for some $p > \nu$. Suppose that for every graded homomorphism $\Phi:\fg\ra\fg'$, the pullback  $\Phi^*\th'_{I'_{\geq 2}}$ is a multiple of $\th_{I_{\geq 2}}$.   Then for every  $\alpha \in \J^kG'$,  we have
\begin{equation}
\label{eqn_d_commutes_pp}
d   (f_P^*\alpha) = f_P^*(d\alpha)
\end{equation}
as distributions, i.e.
$$
\int_Uf_P^*d\al\we\eta=\int_U(-1)^{k+1}f_P^*\al\we d\eta\,,
$$
for every $\eta\in\Om^{N-k-1}_c(U)$. 
\end{theorem}
The hypothesis on graded homomorphisms $\Phi:\fg\ra\fg'$ holds automatically if $\dim G-\dim V_1\leq \dim G'-\dim V_1'$, 
see Remark~\ref{rem_j_condition}; in particular it holds if $G=G'$ or if $G$ and $G'$ are Heisenberg groups of possibly different dimensions.
When $G=G'$ is a step $2$ Carnot group and $\al\in \J^{N-1}G$ has codegree $1$, then Theorem~\ref{thm_j_theorem} follows from a result of Vodopyanov \cite{vodopyanov_foundations}.   Theorem~\ref{thm_j_theorem} has a natural extension to equiregular subriemannian manifolds, see Remark~\ref{rem_j_subriemannian}.

We now consider the $n$-th Heisenberg group $\H_n$ for some $n\geq 1$.  Using the notation from above with $G=\H_n$ we have the grading $\fh=V_1\oplus V_2$, a graded basis $\{X_i\}_{i\in I}$ and its dual $\{\th_i\}_{i\in I}$ where $I_1=\{1,\ldots,2n\}$, $I_2=\{2n+1\}$, and $d\th_{2n+1}=\sum_{j=1}^n\th_{2j-1}\we\th_{2j}$.   Also, for $U\subset \H_n$ open we let $\I^*U$, $\J^*U$ be the differential ideals defined above; note that these do not depend on the group structure on $G$ (i.e. left invariance), but just the subbundle $V_1\subset TU$, which defines a contact structure on $U$.

Adapting the de Rham complex to contact geometry, Rumin \cite{rumin_thesis} has defined a cochain complex 
$$
\calr^0U\stackrel{d_0}{\lra}\calr^1U\stackrel{d_1}{\lra}\ldots\stackrel{d_{2n}}{\lra}\calr^{2n+1}U\stackrel{d_{2n+1}}{\lra} 0
$$
where 
\begin{equation}
\calr^jU=
\begin{cases}
\Om^jU/\I^jU, \quad &0\leq j\leq n\\
\J^jU,\quad &n+1\leq j\leq 2n+1\,,
\end{cases}
\end{equation}
the differential $d_j$ is induced by the exterior derivative if $j\neq n$ and Rumin defines  
$$
d_n:\calr^nU=\Om^nU/\I^nU\ra \J^{n+1}U=\calr^{n+1}U
$$ 
by letting $d_n\al:=d\tilde\al$, where $\tilde\al\in \Om^nU$ is any lift of $\al\in \Om^nU/\I^nU$ such that exterior derivative $d\tilde\al$ belongs to  $\J^{n+1}U$
(it is easy to see that there is a unique $\tilde \alpha$ with this property and that $d_n$ is a second order operator, see \cite[Lemme, p.\  286]{rumin_thesis}).    There is a natural extension of the Rumin complex to a distributional version 
$$
\calr^0_{\cald'}U\stackrel{d_0}{\lra}\calr^1_{\cald'}U\stackrel{d_1}{\lra}\ldots\stackrel{d_{2n}}{\lra}\calr^{2n+1}_{\cald'}U\stackrel{d_{2n+1}}{\lra} 0\,,
$$
where $\calr^*_{\cald'}U$ corresponds to sections with distributional coefficients, and the differentials are defined by duality (see Section~\ref{sec_rumin_complex}). The following result may be viewed as the Heisenberg group analog of Reshetnyak's theorem that pullback by Sobolev mappings commutes with exterior differentiation \cite{reshetnyak_space_mappings_bounded_distortion}.

\begin{theorem}
\label{thm_rumin_chain_mapping}~
Let $U\subset \H_n$ be open, and $f:U\ra \H_n$ be a $W^{1,p}_{\loc}$-mapping for some $p>\nu=2n+2$.  Then Pansu pullback by $f$ induces a chain mapping
\begin{diagram}
\calr^0\H_n&\stackrel{d_0}{\lra}&\calr^1\H_n&\stackrel{d_1}{\lra}\ldots\stackrel{d_{2n}}{\lra}&\calr^{2n+1}\H_n&\stackrel{d_{2n+1}}{\lra} 0\\
\dTo^{f_P^*}&&\dTo^{f_P^*}&&\dTo^{f_P^*}&\\
\calr^0_{\cald'}U&\stackrel{d_0}{\lra}&\calr^1_{\cald'}U&\stackrel{d_1}{\lra}\ldots\stackrel{d_{2n}}{\lra}&\calr^{2n+1}_{\cald'}U&\stackrel{d_{2n+1}}{\lra} 0\,.
\end{diagram}
\end{theorem}
Theorem~\ref{thm_rumin_chain_mapping} -- like Theorem~\ref{thm_j_theorem} --   follows from the Pullback Theorem \cite{KMX1}.

\bigskip
For $U\subset \H_n$ open, we define the {\bf Rumin flat complex} 
$$
\calr^*_{\flat}U\subset \calr^*_{\cald'}U
$$ 
to be the subcomplex consisting of the elements $\al\in  \calr^*_{\cald'}U$ such that both $\al$ and its differential may be represented by $L^\infty$ sections.  This is the Heisenberg analog of the complex $\Om^*_\flat U$ of Whitney flat forms, which is defined for an open subset $U\subset\R^n$, or more generally for a manifold $X$ with a (uniform) bilipschitz structure.  As in the $\R^n$ case, the Rumin flat complex is compatible with bilipschitz homeomorphisms:

\begin{corollary}
\label{cor_bilipschitz_induces_isomorphism}
If $f:U\ra U'$ is a bilipschitz homeomorphism between open subsets of $\H_n$, then $f$ induces an isomorphism of normed cochain complexes
$$
f_P^*:\calr^*_{\flat}U'\lra \calr^*_{\flat}U\,;
$$
the norm of $f_P^*$ and its inverse are bounded depending on the bilipschitz constant of $f$. 
\end{corollary}

By using local charts, this yields a well-defined Rumin flat complex on any bilipschitz contact manifold (or bilipschitz Heisenberg manifold).  More precisely, suppose $L\in [1,\infty)$ and $X$ is a topological manifold equipped with an atlas $\{\phi_i:X\supset\hat U_i\ra U_i\}$ where $U_i\subset\H_n$ and all transition homeomorphisms are all $L$-bilipschitz with respect to the Carnot distance (restricted from $\H_n$).  This will hold, for instance, if $X$ is a metric space locally $L$-bilipschitz to $\H_n$; in particular, any contact manifold of dimension $2n+1$ equipped with a subriemannian metric has this property (in view of Darboux's theorem that contact structures are all locally equivalent up to diffeomorphism).  Then one may define $\calr^*_{\flat}X$ using Corollary~\ref{cor_bilipschitz_induces_isomorphism} and the atlas of charts.  This definition is compatible with locally $L$-bilipschitz embeddings.  Thus one may think of the Rumin flat complex as a replacement for the de Rham complex in the bilipschitz world.   See Section~\ref{sec_rumin_complex} for details.

\section{Pansu pullback and the differential ideal \protect{$\J$}}

In this section we will prove Theorem~\ref{thm_j_theorem}.  We reader may wish to  consult \cite[Sections 3-4]{KMX1} for the results used here.

We retain the notation from the introduction.  Before proceeding with the proof, we recall that there is a notion of {\bf weight} for differential forms.  If we define 
$$
\wt:I\ra \{-1,\ldots,-s\}
$$ by $\wt(i):=-j$ $\iff$ $i\in I_j$, then $\om\in\Om^*U$ has weight $\leq w$ if and only if  it
 is a linear combination with $C^\infty$ coefficients of left invariant forms $\th_{\{i_1\ldots i_k\}}$ where $\sum_j\wt(i_j)\leq w$. 
 We define the weight of $\omega \not \equiv  0$ to be the smallest integer $w$ s.t. $\omega$ has weight  $\le  w$.

Note that if $\om=\sum_{J\subset I}a_J\th_J\in \J^*{U}$, then the condition $\om\we\th_i=0$ for all $i\in I_{\geq 2}$ 
forces $a_J=0$ unless $I_{\geq 2}\subset J$, and hence 
\begin{equation}
\label{eqn_om_in_j_th_i_geq_2}
\om=\om_0\we\th_{I_{\geq 2}}
\end{equation} 
for some $\om_0\in\Om^*U$.  In particular, if $\al\in\J^kU \setminus \{0\}$, then:
\begin{equation}
\label{eqn_coweight_codegree}
\wt(\al)   =-\nu+ N-k\,.
\end{equation}

\begin{remark}
\label{rem_i_j_tensor_product_description}
Letting $\La^*\fg\subset \Om^*G$ denote the subcomplex of left invariant forms on $G$, (which we identify with the exterior algebra on $\fg^*$), we obtain differential ideals of $\La^*\fg$ by intersecting with $\I^*$, $\J^*$:
$$
\I^*\fg:=\I^*G\cap\La^*\fg\, ,\quad\J^*\fg:=\J^*G\cap \La^*\fg\,
$$  
where $\J^*\fg$ is also the annihilator of $\I^*\fg$ in $\La^*\fg$.  Then we have the alternate descriptions
$$
\I^*U\simeq C^\infty (U)\otimes \I^*\fg\,,\quad \J^*U\simeq C^\infty (U)\otimes \J^*\fg\,.
$$
\end{remark}

\bigskip

\begin{proof}[Proof of Theorem~\ref{thm_j_theorem}]

The theorem is immediate if the pullbacks $f_P^*\al$, $f_P^*d\al$ vanish almost everywhere, so we assume this is not the case.  Since $\al=\al_0\we\th'_{I'_{\geq 2}}$, $d\al=\be_0\we\th'_{I'_{\geq 2}}$ for some $\al_0\,,\be_0\in \Om^*G'$ by \eqref{eqn_om_in_j_th_i_geq_2}, we get that $f_P^*\th'_{I'_{\geq 2}}(x)\neq 0$ for some point of Pansu differentiability $x\in U$.  From the assumptions of the theorem we therefore have $(D_Pf(x))^*\th'_{I'_{\geq 2}}(f(x))=\la \th_{I_{\geq 2}}$ for some $\la\neq 0$.  Since $D_Pf(x):\fg\ra \fg'$ is a graded homomorphism
it follows that 
$$
-\nu+\dim V_1=\wt \th_{I_{\geq 2}}=\wt\th'_{I'_{\geq 2}}=-\nu'+\dim V_1'\,.
$$
Combining this with \eqref{eqn_coweight_codegree} we obtain
$$
\wt(\al)\leq -\nu+(N-k)\,,\quad \wt(d\al)\leq -\nu+N-(k+1)\,.
$$
For all $\eta\in \Om^{\ell}_c(U)$ with $\ell=N-(k+1)$ we have 
$$
\wt(\eta)\leq -\ell= -(N-(k+1))\,,\quad \wt(d\eta)\leq -\ell-1=-(N-k)\,;
$$ 
and hence
\begin{align*}
\wt(d\al)+\wt(\eta)\leq -\nu\\
\wt(\al)+\wt(d\eta)\leq -\nu\,.
\end{align*}
By the Pullback Theorem \cite[Theorem 1.5]{KMX1} we have
$$
\int_Uf_P^*d\al\we\eta=\int_U(-1)^{k+1}f_P^*\al\we d\eta\,,
$$
which gives \eqref{eqn_d_commutes_pp}. 
\end{proof}

\bigskip\bigskip
\begin{remark}
\label{rem_j_condition}
Note that Theorem~\ref{thm_j_theorem} holds nontrivially only when there is a graded homomorphism $\Phi:\fg\ra\fg'$ which induces an isomorphism $\sum_{j\geq 2}V_j\ra \sum_{j\geq 2}V'_j$.  In this case we have decomposition $\fg=\ker\Phi\oplus\fg_0$ into graded ideals where $\ker\Phi\subset V_1$ is an abelian ideal, and $\Phi$ induces an embedding $\fg_0\hookrightarrow \fg'$ onto an ideal containing $\sum_{j\geq 2}V'_j$.   This happens, for instance, when $G=\H_m$, $G'=\H_{m'}$ are Heisenberg groups with $m\leq m'$.
\end{remark}

\begin{remark}
\label{rem_j_subriemannian}
 Theorem~\ref{thm_j_theorem} extends in a straightforward way to $W^{1,p}$-mappings between equiregular subriemannian manifolds using the generalization of the pullback theorem in \cite[Appendix A]{KMX1}.  We indicate briefly how the setup and proof generalize.  Let $(M,\{V_j\}_{1\leq j\leq s},g)$, $(M',\{V_j'\}_{1\leq j\leq s'},g')$ be equiregular subriemannian manifolds, i.e. $M$ is a smooth manifold, $V_1\subsetneq\ldots \subsetneq V_s=TM$ is collection of subbundles where $V_j:=[V_1,V_{j-1}]$ for $1<j\leq s$, and $g$ is a Riemannian metric on $M$, and similarly for $(M',\{V_j'\}_{1\leq j\leq s'},g')$.  We let $\nu=\sum_jj\dim(V_j/V_{j-1})$ denote the homogeneous dimension of $M$.  Then the definition of the differential ideals in \eqref{eqn_def_j} extends  verbatim to yield differential ideals $\I^* M'\,,\J^* M'\subset\Om^*M$.   The hypothesis of Theorem~\ref{thm_j_theorem} on graded homomorphisms generalizes to an assertion about graded homomorphisms between nilpotent tangent cones.  Under this assumption, if $f:M\ra M'$ is a $W^{1,p}$-mapping for some $p>\nu$ then  the Pansu pullback is well-defined and commutes with the exterior derivative for forms $\om\in \J^*M'$.  To verify this, since the assertion of the theorem may be localized without loss of generality one may work with adapted (co)frames as in \cite[Appendix A]{KMX1}; the rest of the proof follows the same lines as in the Carnot group case, beginning with \eqref{eqn_om_in_j_th_i_geq_2}. (Note that for for general differential forms, the definition Pansu pullback used in \cite[Appendix A]{KMX1} depends on the choice of metrics $g$, $g'$; nonetheless, for a form $\om\in \J^*M'$, the pullback is independent of the choice of $g$, $g'$ due to \eqref{eqn_om_in_j_th_i_geq_2}.) 
\end{remark}

\bigskip\bigskip
\section{Pansu pullback and the Rumin complex}
\label{sec_rumin_complex}
In this section we prove that the Rumin complex is respected by Sobolev mappings.  We first characterize Rumin's differential weakly, and then use this to define distributional and (locally) flat versions of the Rumin complex.  We then prove Theorem~\ref{thm_rumin_chain_mapping}, and apply it to prove the bilipschitz invariance of the Rumin flat complex.
Fix $n\geq 1$, and let $U\subset\H_n$ be an open subset.

\bigskip
\subsection*{The weak formulation of the Rumin differential}~
Before proceeding, we make some remarks about the Rumin complex.

Note that the wedge product descends to a bilinear mapping 
\begin{equation}
\label{eqn_pairing_descends}
\Om^kU/\I^kU\times \J^\ell U\stackrel{\we}{\lra}\Om^{k+\ell}U\,,
\end{equation} 
since $\J^\ell U$ is contained in the annihilator of $\I^kU$.   We let $\calr^*_cU\subset\calr^*U$ denote the subcomplex of elements with compact support, i.e. $\calr^k_cU:=\Om^k_cU/\I^k_cU$ for $0\leq k\leq n$ and $\calr^k_cU:=\J^k_cU$ for $k\geq n+1$.  Letting $\I^*\fh_n\subset\La^*\fh_n$ denote the left invariant subspace of $\Om^*\H_n$, we have isomorphisms 
\begin{align*}
\calr^k_cU:=&\Om^k_cU/\I^k_cU\simeq (C^\infty_c(U)\otimes \La^k\fh_n)/(C^\infty_c(U)\otimes\I^k\fh_n)\\
\simeq& C^\infty_c(U)\otimes(\La^k\fh_n/\I^k\fh_n)\simeq C^\infty_c(U;\La^k\fh_n/\I^k\fh_n)\\
\simeq& \Ga_{C^\infty_c}(\La^kTU/\I^kTU)
\end{align*}
where $\La^kTU/\I^kTU$ is the quotient of the two bundles $\La^kTU$ and $\I^kTU$ and $\Ga_{C^\infty_c}$ refers to smooth sections with compact support.  By choosing a complement $W\subset \La^k\fh_n$ to $\I^k\fh_n$, for any $0\leq k\leq n$ we may lift any $\al\in \calr^k_cU\simeq C^\infty_c(U;\La^k\fh_n/\I^k\fh_n)$ to $\hat\al\in C^\infty_c(U,W)\subset C^\infty_c(U,\La^k\fh_n)$ such that $\spt\hat\al=\spt\al$.

We now show that the differentials of the Rumin complex may be characterized weakly as dual to  the exterior derivative, appropriately interpreted (using the bilinear pairing defined by wedge product and integration).

\begin{lemma}
\label{lem_weak_characterization_rumin_differentials}~
For $0\leq k<2n+1$, if $\be\in \calr^kU$ and $\ga\in \calr^{k+1}U$, then $\ga=d_k\be$ if and only if 
\begin{equation}  \label{eq:weak_rumin_all}
\int_U   \beta \wedge d_{2n-k} \eta =(-1)^{k+1}  \int_U   \gamma \wedge \eta  
\end{equation}
for every $\eta\in \calr^{2n-k}_cU$; here the wedge products refer to the pairing \eqref{eqn_pairing_descends}.  More explicitly, the following hold.
\ben
\item Suppose $\be\in \Om^kU/\I^kU$ and $\ga\in \Om^{k+1}U/\I^{k+1}U$ for some $0\leq k<n$.  Then $\ga=d_k\be$ if and only if 
\begin{equation}  \label{eq:weak_rumin_low}
\int_U   \beta \wedge d \eta =(-1)^{k+1}  \int_U   \gamma \wedge \eta  
\end{equation}
for every $ \eta \in J^{2n-k}_cU$.
\item Suppose $\be\in \Om^nU/\I^nU$ and $\ga\in \J^{n+1}U$.  Then $\ga=d_n\be$ if and only if for every $ \eta\in \Om^n_cU/\I^n_cU$
we have 
\begin{equation}  
\label{eq:weakD}
\int_U \beta \wedge d_n \eta \;
=(-1)^{n+1}  \int_U \gamma \wedge \eta \,;
\end{equation}
\item Suppose $\be\in \J^k$ and $\ga\in \J^{k+1}$ for some $n< k<2n+1$.      Then $\ga=d_k\be$ if and only if 
\begin{equation} \label{eq:weak_rumin_high}
\int_U\be\we d\eta=(-1)^{k+1}\int_U\ga\we \eta
\end{equation}
for every $\eta\in \Om^{2n-k}_cU$. 
\een
\end{lemma}

\bigskip\bigskip
Below we will identify $\La^*V_1$ with a subspace of $\La^*\fh_n$ using the embedding induced by projection $\fh_n=V_1\oplus V_2\lra V_1$.   By convention we let $\La^jV_1=\{0\}$ for all $j<0$.

For the proof we will need the following algebraic fact.

\bigskip
\begin{lemma}
\label{lem_quotient_j_duality}~
\ben
\item The map $W_k:\La^kV_1\ra\La^{k+2}V_1$  given by $W_k(\al):=d\th_{2n+1}\we\al$ is injective for $k\leq n-1$ and surjective for $k\geq n-1$. 
\item 
If      $\I^*\fh_n,\J^*\fh_n\subset \La^*\fg$ denote the left-invariant subspaces of $\I^*\H_n$, $\J^*\H_n$, respectively, then wedge product induces a nondegenerate pairing
$$
\La^k\fh_n/\I^k\fh_n\times \J^{2n+1-k}\fh_n\stackrel{\we}{\lra}\La^{2n+1}\fh_n\,.
$$
for $0\leq k\leq n$. 
\een  
\end{lemma}

\bigskip
\begin{proof}~
(1).  Since the $2$-form $d\th_{2n+1}$ restricts to a symplectic form on $V_1$, this is a classical fact \cite{rumin_thesis,weil_introduction_etude_varietes_kahleriennes}.

(2).   The pairing is well defined by the definitions of $\I^*\h_n$ and $\J^*\h_n$.

Note that 
\begin{align*}
\La^k\fh_n&=\{\th_{2n+1}\we\al+\be\mid\al\in \La^{k-1}V_1\,,\;\be\in\La^kV_1\}\,,\\
\I^k\fh_n&=\{\th_{2n+1}\we\al+d\th_{2n+1}\we\be\mid\al\in\La^{k-1}V_1\,,\;\be\in\La^{k-2}V_1\},\\
\J^\ell\fh_n&=\{\th_{2n+1}\we\al\mid\al\in\La^{\ell-1}V_1\,,\;d\th_{2n+1}\we\al=0
\}\,.
\end{align*}
Hence the inclusion induces an isomorphism
$$
\La^kV_1/\im W_{k-2}\stackrel{\sim}{\lra}\La^k\fh_n/\I^k\fh_n\,.
$$ 
Therefore we are reduced to showing that  the wedge product
induces a nondegenerate pairing
\begin{equation}
\label{eqn_reduced_pairing}
\La^kV_1/\im W_{k-2}\;\times\; \ker W_{2n-k}
\stackrel{\we}{\lra}\La^{2n}V_1\,.
\end{equation}
Wedge product induces an isomorphism 
$$
\La^kV_1\stackrel{\sim}{\lra}L(\La^{2n-k}V_1,\La^{2n}V_1)\,,
$$
so composing with the restriction map 
$$
L(\La^{2n-k}V_1,\La^{2n}V_1)\lra L(\ker W_{2n-k},\La^{2n}V_1)
$$
we obtain a surjection
\begin{equation}
\label{eqn_we_surjection}
\La^kV_1/\Im W_{k-2}\lra L(\ker W_{2n-k},\La^{2n}V_1)\,.
\end{equation}
Using (1) and the fact that $\dim\La^jV_1=\dim\La^{2n-j}V_1$ for all $j$, we have
\begin{align*}
\dim\La^kV_1/\Im W_{k-2}&=\dim\La^kV_1-\dim \La^{k-2}V_1\\
&=\dim\La^{2n-k}V_1-\dim\La^{2n-(k-2)}V_1\\
&=\dim\ker W_{2n-k}\\
&=\dim L(\ker W_{2n-k},\La^{2n}V_1)\,.
\end{align*}
Hence \eqref{eqn_we_surjection} is an isomorphism and it follows that \eqref{eqn_reduced_pairing} is nondegenerate.
\end{proof}

\bigskip\bigskip
\begin{proof}[Proof of Lemma~\ref{lem_weak_characterization_rumin_differentials}]~
All three cases follow from integration by parts and the duality in Lemma~\ref{lem_quotient_j_duality}(2).

(1). Integrating by parts, we have
$$
\int_U\be\we d\eta+(-1)^k\ga\we\eta=(-1)^k\int_U(\ga-d_k\be)\we\eta\,;
$$
by Lemma~\ref{lem_quotient_j_duality}(2) this vanishes for every $\eta\in \J^{2n-k}_cU$ if and only if $\ga=d_k\be$.

(2). Choose lifts $\tilde\eta\in\Om^n_cU$,  $\tilde\be\in \Om^nU$ of $\eta$ and $\be$ respectively, such that $d\tilde\eta\,,d\tilde\be\in \J^{n+1}U$; hence $d\tilde\eta=d_n\eta$ and  $d\tilde\be=d_n\be$.  Integrating by parts we have
\begin{align*}
\int_U&\be\we d_n\eta+(-1)^n\ga\we\eta\\
&=\int_U\tilde\be\we d\tilde\eta+(-1)^n\ga\we\tilde\eta\\
&=(-1)^n\int_U(\ga-d\tilde\be)\we\tilde\eta\\
&=(-1)^n\int_U(\ga-d_n\be)\we \eta\,;
\end{align*}
this vanishes for every $\eta\in \Om^n_cU/\I^n_cU$ if and only if $\ga=d_n\be$, by Lemma~\ref{lem_quotient_j_duality}(2).

(3). Similar to (1).

\end{proof}

\bigskip\bigskip
\subsection*{The distributional Rumin complex}
Let $U\subset\H_n$ be an open subset.  Note that the map
$$
\calr^kU\ra L(\calr^{2n+1-k}_cU,\R)
$$
which sends $\al\in\calr^kU$ to the linear functional 
$$
\eta\mapsto \int_U\al\we\eta
$$
is an embedding by Lemma~\ref{lem_quotient_j_duality}(2).
Motivated by this and by Lemma~\ref{lem_weak_characterization_rumin_differentials}, we define the distributional version $\calr^*_{\cald'}U$ of the Rumin complex as follows.  We let $\calr^k_{\cald'}U$ be the collection of linear functionals $\al:\calr^{2n+1-k}_cU\ra \R$ satisfying the (standard distribution-type continuity) condition that for every compact subset $K\subset U$, there exist $j\in\Z$, $C\in\R$ such that for every $\eta\in \calr^{2n+1-k}_cU$ supported in $K$ we have
$$
|\al(\eta)|\leq C\,\|\eta\|_{C^j}\,,
$$ 
where the $C^j$ norm is defined using the isomorphism 
\begin{equation*}
\calr^k_cU\simeq 
\begin{cases}
C^\infty_c(U;\La^k\fh_n/\I^k\fh_n)\quad 0\leq k\leq n\\
C^\infty_c(U;\J^k\fh_n)\quad k\geq n+1\,,
\end{cases}
\end{equation*}
and differentiation using left invariant vector fields.   
(An alternative definition is $\cald'U\otimes \calr^*\fh_n$, where $\calr^*\fh_n\subset\calr^*\H_n$ is the subcomplex of left-invariant elements; this is equivalent, since  (after fixing a nonzero element $\xi\in \La^{2n+1}\fh_n$) we have an isomorphism  $\cald'U\otimes \calr^*\fh_n\stackrel{\sim}{\ra}\calr^*_{\cald'}U$ which sends $u\otimes\bar\om\in \cald'U\otimes\calr^k\fh_n$ to the linear functional 
$$
\phi\bar\eta\mapsto u(\phi)(\bar\om\we\bar\eta)(\xi)
$$
for $\phi\in C^\infty_cU$, $\bar\eta\in \La^{2n+1-k}\fh_n$.)
Following \eqref{lem_weak_characterization_rumin_differentials} the differential $d_k:\calr^k_{\cald'}U\ra \calr^{k+1}_{\cald'}U$ is defined by
$$
d_k\be(\eta)=(-1)^{k+1}\be(d_{2n-k}\eta)\,.
$$

Now consider a $W^{1,p}_{\loc}$-mapping $f:\H_n\supset U\ra \H_n$ for some $p>2n+1$.  If $x\in U$ is a point of Pansu differentiability, then the graded homomorphism $D_Pf(x):\fh_n\ra\fh_n$ induces linear mappings $\calr^k\fh_n\ra \calr^k\fh_n$.  (To see this, note that $\Phi^*(\I^*\fh_n)\subset \I^*\fh_n$, so $\Phi^*:\calr^k\fh_n=\La^k\fh_n/\I^k\fh_n\ra \La^k\fh_n/\I^k\fh_n$ is well-defined for $0\leq k\leq n$.  If $n+1\leq k\leq 2n+1$, then $\Phi^*(\J^k)=\J^k$ if $\Phi$ is  a graded isomorphism because $\J^*:=(\I^*)^\perp$, while $\Phi^*(\J^k)=\{0\}$ if $\Phi$ is not an isomorphism.)  By the Sobolev condition, we get that $f_P^*\om\in L^1_{\loc}$ for every $\om\in \calr^*\H_n$, and hence we have a mapping $f_P^*:\calr^*\H_n\ra \calr^*_{\cald'}U$.

\bigskip\bigskip
\begin{proof}[Proof of Theorem~\ref{thm_rumin_chain_mapping}]
The theorem reduces to the assertion that for every $0\leq k<2n+1$, $\al\in \calr^k\H_n$, $\eta\in \calr^{2n-k}_cU$
\begin{equation}
\label{eqn_f_p_al_eta}
\int_Uf_P^*\al\we d_{2n-k}\eta=(-1)^{k+1}\int_Uf_P^*d_k\al\we\eta\,.
\end{equation}
For any value of $k$ we may choose lifts of $\al$ or $\eta$ (as appropriate) such that \eqref{eqn_f_p_al_eta} follows directly from the Pullback Theorem \cite{KMX1}. 
When $k\geq n$, we may choose the lift $\tilde\eta\in\Om^{2n-k}U$ 
of $\eta\in \calr^{2n-k}_c$ to have compact support; this follows from Rumin's construction 
when $k=n$, and when $k>n$ it may be arranged by choosing a complement $W\subset\La^{2n-k}\fh_{2n-k}$ 
to the subspace $\I^{2n-k}\fh_{2n-k}\subset\La^{2n-k}\fh_{2n-k}$ and requiring that $\tilde\eta(x)\in W$ for all $x\in U$.
\end{proof}

\bigskip\bigskip
For $1<p\leq \infty$ we  define $\calr^*_{L^p}U$  to be the collection of all  $\al\in \calr^*_{\cald'}U$ which define a bounded linear functional on $\calr^*U$ equipped with the $L^{p^*}$-norm, where $p^*$ is the exponent conjugate to $p$;    this is equivalent to taking $\calr^*_{L^p}U:=L^pU\otimes \calr^*\fh_n$.  The {\bf Rumin flat complex $\calr^*_\flat U$} is then defined to be the set of $\al\in\calr^*_{L^\infty}U$ such that $d\al\in \calr^*_{L^\infty}U$; this is a normed cochain complex with the norm $\|\al\|_\flat:=\max(\|\al\|_{L^\infty}\,,\|d\al\|_{L^\infty})$.  Similarly, we may define local versions $\calr^*_{L^p_{\loc}}U\subset\calr^*_{\cald'}U$ and $\calr^*_{\flat_{\loc}}U\subset\calr^*_{\cald'}U$.

\bigskip\bigskip
\begin{proof}[Proof of Corollary \ref{cor_bilipschitz_induces_isomorphism}]
This follows from Theorem~\ref{thm_rumin_chain_mapping} by an approximation argument.

Let $U':=f(U)$ and pick $\om\in \calr^k_\flat U'$.  By mollifying (using left translation),
we may find a sequence $U_j'\subset U'$ of open subsets and $\om_j\in \calr^kU'_j$ such that:
\bit
\item The $U'_j$ exhaust $U$.
\item $\sup_j\max(\|\om_j\|_{L^\infty}\,,\|d_k\om_j\|_{L^\infty})<\infty$.    
\item $\om_j(x)\ra \om(x)$, $d_k\om_j(x)\ra d_k\om(x)$ for a.e. $x\in U'$.
\eit
Since $f$ is bilipschitz, it induces an isomorphism of Banach spaces $f_P^*:\calr^k_{L^\infty}U'\ra\calr^k_{L^\infty}U$, and moreover $f_P^*\om_j(x)\ra f_P^*\om(x)$ and $(f_P^*d_k\om_j)(x)\ra (f_P^*d_k\om)(x)$ for a.e. $x\in U$.
If $\eta\in \calr^{2n+1-k}_cU$, then $(d_kf_P^*\om_j)(\eta)=(f_P^*d_k\om_j)(\eta)$ by Theorem~\ref{thm_rumin_chain_mapping}.  Hence by the Dominated Convergence Theorem
\begin{align*}
(d_kf_P^*\om)(\eta)&=\lim_{j\ra\infty}(d_kf_P^*\om_j)(\eta)\\
&=\lim_{j\ra\infty}(f_P^*d_k\om_j)(\eta)\\
&=(f_P^*d_k\om)(\eta)\,.
\end{align*}
This shows that $f_P^*$ induces a chain mapping $\calr^*_\flat U'\ra \calr^*_\flat U$.   Applying the same reasoning to $(f^{-1})^*_P$ yields the corollary.

\end{proof}

\bigskip\bigskip
\subsection*{The Rumin complex on bilipschitz Heisenberg manifolds}
We now use Corollary~\ref{cor_bilipschitz_induces_isomorphism} to construct a bilipschitz natural chain complex on manifolds.
\begin{definition}
A metric space $X$ is a {\bf bilipschitz Heisenberg manifold} if for some $n$ it is locally bilipschitz homeomorphic to $\H_n$; if the bilipschitz constants may be taken to be uniform, then $X$ is a {\bf uniformly bilipschitz Heisenberg manifold}.
\end{definition}
Now suppose $X$ is locally $L$-bilipschitz homeomorphic to $\H_n$. Then we may find a collection 
$$
\{\phi_i:X\supset U_i\ra\hat U_i\subset \H_n\}_{i\in I}
$$ 
of $L$-bilipschitz homeomorphisms between open sets, such that $\{U_i\}_{i\in I}$ is an open cover of $X$.  Imitating the chartwise definition of tensor fields on smooth manifolds, we may define the complex $\calr^*_{\flat}X$ by letting elements be collections $\{\al_i\in \calr^*_\flat\hat U_i\}_{i\in I}$ such that:
\bit
\item $\sup_i\|\al_i\|_{L^\infty}<\infty$.
\item The $\al_i$s are compatible with the transition homeomorphims, i.e. for every $i,j\in I$,  the transition homeomorphism   
$$
\phi_j\circ\phi_i^{-1}:\phi_i(U_i\cap U_j)\ra \phi_j(U_i\cap U_j)
$$
which has bilipschitz constant $L^2$, pulls back  $\al_j\restr_{\phi_j(U_i\cap U_j)}$ to $\al_i\restr_{\phi_i(U_i\cap U_j)}$.
\eit

If we only know that $X$ is locally bilipschitz equivalent to $\H_n$ without uniform control on the bilipschitz constant, then the construction above may be modified in a straightforward way to obtain a locally flat version $\calr^*_{\flat_{\loc}}X$ of the complex $\calr^*_{\flat}X$.  Corollary~\ref{cor_bilipschitz_induces_isomorphism} now gives:

\begin{corollary}
\label{cor_bilipschitz_induces_isomorphism_heisenberg_manifolds}
If $f:X\ra X'$ is a bilipschitz homeomorphism between bilipschitz Heisenberg manifolds, then Pansu pullback induces an isomorphism of chain complexes $\calr^*_{\flat_{\loc}}X'\lra \calr^*_{\flat_{\loc}}X$.
\end{corollary}

\bibliography{product_quotient}

\providecommand{\bysame}{\leavevmode\hbox to3em{\hrulefill}\thinspace}
\providecommand{\MR}{\relax\ifhmode\unskip\space\fi MR }
\providecommand{\MRhref}[2]{%
  \href{http://www.ams.org/mathscinet-getitem?mr=#1}{#2}
}
\providecommand{\href}[2]{#2}
\begin{thebibliography}{KMX20}

\bibitem[KMXa]{kmx_approximation_low_p}
B.~Kleiner, S.~M\"{u}ller, and X.~Xie, \emph{Pansu pullback and exterior
  differentiation for {S}obolev maps on {C}arnot groups}, arXiv:2007.06694.

\bibitem[KMXb]{KMX2}
\bysame, \emph{Pansu pullback and rigidity of mappings between {C}arnot groups
  {II}}.

\bibitem[KMX20]{KMX1}
\bysame, \emph{Pansu pullback and rigidity of mappings between {C}arnot groups
  {I}}, 2020.

\bibitem[MM95]{margulis_mostow_differential_quasiconformal_mapping}
G.~A. Margulis and G.~D. Mostow, \emph{The differential of a quasi-conformal
  mapping of a {C}arnot-{C}arath\'{e}odory space}, Geom. Funct. Anal.
  \textbf{5} (1995), no.~2, 402--433. \MR{1334873}

\bibitem[Pan89]{pansu}
P.~Pansu, \emph{M\'{e}triques de {C}arnot-{C}arath\'{e}odory et
  quasiisom\'{e}tries des espaces sym\'{e}triques de rang un}, Ann. of Math.
  (2) \textbf{129} (1989), no.~1, 1--60. \MR{979599}

\bibitem[Res89]{reshetnyak_space_mappings_bounded_distortion}
Yu.~G. Reshetnyak, \emph{Space mappings with bounded distortion}, Translations
  of Mathematical Monographs, vol.~73, American Mathematical Society,
  Providence, RI, 1989, Translated from the Russian by H. H. McFaden.
  \MR{994644}

\bibitem[Rum94]{rumin_thesis}
Michel Rumin, \emph{Formes diff\'{e}rentielles sur les vari\'{e}t\'{e}s de
  contact}, J. Differential Geom. \textbf{39} (1994), no.~2, 281--330.
  \MR{1267892}

\bibitem[Vod03]{vodopyanov_differentiability_2003}
S.~K. Vodopyanov, \emph{On the differentiability of mappings of {S}obolev
  classes on the {C}arnot group}, Mat. Sb. \textbf{194} (2003), no.~6, 67--86.
  \MR{1992177}

\bibitem[Vod07]{vodopyanov_foundations}
\bysame, \emph{Foundations of the theory of mappings with bounded distortion on
  {C}arnot groups}, The interaction of analysis and geometry, Contemp. Math.,
  vol. 424, Amer. Math. Soc., Providence, RI, 2007, pp.~303--344. \MR{2316342}

\bibitem[Wei58]{weil_introduction_etude_varietes_kahleriennes}
Andr\'{e} Weil, \emph{Introduction {\`a} l'\'{e}tude des vari\'{e}t\'{e}s
  k\"{a}hl\'{e}riennes}, Publications de l'Institut de Math\'{e}matique de
  l'Universit\'{e} de Nancago, VI. Actualit\'{e}s Sci. Ind. no. 1267, Hermann,
  Paris, 1958. \MR{0111056}

\end{thebibliography}
\bibliographystyle{amsalpha}

\end{document}